\documentclass[10pt]{article}
\usepackage{mathtools}
\usepackage{graphicx}
\usepackage{amsmath,amsfonts,latexsym,amscd,amssymb,theorem}
\usepackage{graphicx}
\usepackage[english]{babel}
\usepackage[shortlabels]{enumitem}
\usepackage[ansinew]{inputenc}
\usepackage{epsfig}
 \usepackage{epigraph}
\usepackage{amsmath}
\usepackage{amssymb}
\usepackage{romannum}
\usepackage{afterpage}
\usepackage{selinput}
\usepackage{fancyhdr}
\newcommand{\ba}{\begin{eqnarray}}
\newcommand{\ea}{\end{eqnarray}}

\newtheorem{thm}{Theorem}[section]
\newtheorem{problem}{Problem}
\newtheorem{conjecture}{Conjecture}

\newtheorem{theorem}[thm]{Theorem}

\newtheorem{lemma}[thm]{Lemma}
 \newtheorem{assertion}[thm]{Assertion}
  
\newtheorem{proposition}[thm]{Proposition}
\newtheorem{corollary}[thm]{Corollary}

\makeatletter
\newcommand*{\rom}[1]{\expandafter\@slowromancap\romannumeral #1@}
\makeatother

\newcommand\blfootnote[1]{%
  \begingroup
  \renewcommand\thefootnote{}\footnote{#1}%
  \addtocounter{footnote}{-1}%
  \endgroup
}

\date{}
\begin{document}
\title{Subdivisions of Six-Blocks Cycles $C(k,1,1,1,1,1)$ in  Strong Digraphs}
 	\maketitle
 	\begin{center}\author{Hiba Ayoub$^{1}$, Soukaina Zayat}\footnote[1]{KALMA Laboratory, Department of Mathematics, Faculty of Sciences I, Lebanese University, Beirut - Lebanon}$^{,}$\footnote[2]{Arts, Sciences, and Technology University in Lebanon, CRAMS: Center for Research in Applied Mathematics and Statistics, Faculty of Sciences, Beirut, Lebanon.}, \author{Darine Al-Mniny}\footnote[4]{KALMA Laboratory, Department of Mathematics, Lebanese University, Baalbeck - Lebanon.}$^{,}$\footnote[5]{Department of Mathematics and Physics,  Lebanese International University LIU, Rayak, Lebanon.}$^{,}$\footnote[6]{ Department of Mathematics and Physics,  The International University of Beirut BIU, Beirut, Lebanon.}$^{,}$\footnote[7]{ Department of Computer Science and  Mathematics,  Lebanese American University, Beirut, Lebanon.} \end{center} 
 \begin{abstract}
 
\blfootnote{\textbf{Corresponding Author:}\\
	Soukaina Zayat, KALMA Laboratory, Department of Mathematics, Faculty of Sciences I, Lebanese University, Beirut - Lebanon\\
Email: soukaina.zayat.96@outlook.com} 
 
 	\noindent A cycle $C(k_{1}, k_{2}, \dots,k_{n})$ is the oriented cycle formed of $n$ blocks of lengths $k_{1}, k_{2}, \dots,k_{n-1}$ and $k_{n}$ respectively. In 2018 Cohen et al.\ conjectured that for every positive integers $k_1, k_2,\dots,k_n$, there exists a constant $g(k_1,k_2,\dots,k_n)$ such that every strongly connected digraph  containing no subdivisions of $C(k_1, k_2, \dots, k_n)$ has a  chromatic number at most $g(k_1,k_2,\dots,k_n)$. 
 	In their paper, Cohen et al.\ confirmed the conjecture for cycles with two blocks and for cycles with four blocks having all its blocks of length $1$. Recently, the conjecture was proved for special types of four-blocks cycles. 
 	In this paper, we confirm Cohen et al.\ 's conjecture for all six-blocks cycles $C(k,1,1,1,1,1)$. Precisely, for any integer $k$, we prove that every strongly connected digraph  containing no subdivisions of $C(k,1,1,1,1,1)$ has a  chromatic number at most $O(k)$, and we significantly reduce the chromatic number in case $k=1$. 
 	\end{abstract}
 	\textbf{Keywords.} Chromatic number, strong digraph, subdivision,  six-blocks cycle.
 	\section{Introduction}
 	\pagenumbering{arabic}
 	A graph $G$  is said to be simple if it contains no loops and no multiple edges. All graphs in this paper are considered simple. The chromatic number of a graph $G$, denoted by $ \chi(G)$, is the minimal number of colors needed to color the vertices of $G$, such that any two adjacent vertices have distinct colors. We call $D$  a digraph, if it is obtained from a graph $G$ by giving an orientation to each edge. In this case, $G$, denoted by $G(D)$, is called the underlying graph of  $D$. Denote by $V(D)$ the set of vertices of $D$, and by $A(D)$ the set of arcs of $D$. The chromatic number of a digraph $D$ is the chromatic number of it's underlying graph.  If $\chi(D)=k$, then we say that $D$ is $k$-chromatic .   \medbreak

\noindent  The length of a path  (resp. cycle) $P$, denoted by $l(P)$ is the number of its edges.  A path $P=x_1x_2 \dots x_t$ is said to be directed, if $(x_i,x_{i+1}) \in A(P)$, for all $i \in \lbrace 1, \dots t-1 \rbrace$.  A digraph $D$ is said to be strongly connected, if for any two vertices $x$ and $y$ there exists a directed path from $x$ to $y$. A block of a path $P$ (resp. cycle) is a maximal directed subpath  of $P$. Denote by $P(k_{1}, k_{2}, \dots,k_{t})$ (resp. $C(k_{1}, k_{2}, \dots,k_{t})$) the oriented path (resp. cycle) formed of $n$ blocks of lengths $k_{1}, k_{2}, \dots,k_{t-1}$ and $k_{t}$ respectively. \medbreak 

\noindent A subdivision of $D$, denoted by $S$-$D$,  is a digraph obtained from $D$ by replacing each arc $(x,y)$ of $D$  by a directed path from $x$ to $y$ of length at least $1$, where all new paths are internally disjoint. A digrah is said to be $D$-subdivision-free if it contains no subdivision of $D$ as a subdigraph.\\

The following problem is proposed by Cohen et al.\ \cite{hjk} in 2018:
 \begin{problem}\label{problem2} What are the digraphs $D$ that can be found in any $k$-chromatic digraph as subdigraphs?\end{problem}

\noindent In \cite{lh} Erd\"{o}s  proved the existence of digraphs with  arbitrarily high girth and arbitrarily large chromatic number. Which means the existence of  digraphs with arbitrarily high chromatic number containing no subdigraph isomorphic to a digraph $D$ containing an oriented cycle. So, the only possible way to answer Problem \ref{problem2}  is the oriented trees.  Burr \cite{kg} conjectured that every $(2k-2)$-chromatic digraph contains every oriented tree $T$ of order $k$, and proved that every $(k-1)^2$-chromatic digraph contains a copy of any oriented tree $T$ of order $k$. The best known  bound, obtained by Addario-Berry et al.\ \cite{jhp}, is  $(k/2)^2$. For special oriented trees, better bounds on the chromatic number are known. The most celebrated one, known as Gallai-Hasse-Roy-Vitaver theorem \cite{Gallai, Roy} states that every $n$-chromatic digraph contains a directed path of length $n-1$. El Sahili \cite{sahili1} conjectured that every $n$-chromatic digraph contains every oriented path of order $n$. Addario et al.\ \cite{khj} and El Sahili et al.\ \cite{sara}, proved independently that this conjecture holds for two blocks paths with $n \geq 4$. More results concerning paths can be found in \cite{batoul1, batoul2, maydoun, jobbe}.
  Bondy \cite{fdo} proved that  every strong digraph $D$ contains a directed cycle of length at least $\chi(D)$. The strong connectivity condition is necessary, because of the existence of  acyclic digraphs (transitive tournaments) with no directed cycle and large chromatic number. As any directed cycle of length at least $k$ is a subdivision of the directed cycle $C_{k}$   of length $k$,   Cohen et al.\  conjectured extension of Bondy's theorem to all oriented cycles, as follows:
  \begin{conjecture}(Cohen et al.\ \cite{hjk})\label{conj} For every positive integers $k_1, k_2,\dots,k_n$, there exists a constant $g(k_1,k_2,\dots,$ $k_n)$ such that the chromatic number of every strongly connected digraph  with no subdivisions of the oriented cycle  $C(k_1, k_2, \dots, k_n)$ is at most $g(k_1,k_2,\dots,k_n)$.\end{conjecture}
   
\noindent The following theorem shows that the strong connectivity assumption is necessary in Conjecture \ref{conj}: 
\begin{theorem}(Cohen et al.\ \cite{hjk})
For any positive integers $b,c$, there exists an acyclic digraph $D$ with $\chi(D) \geq c$ in which all oriented cycles have more than $b$ blocks.
\end{theorem}

\noindent   In  \cite{hjk} Cohen et al.\  proved that the chromatic number of strong digraphs with no subdivisions of a two-blocks cycle $C(k_{1},k_{2})$ is at most   $O((k_{1}+k_{2})^4)$, which confirms Conjecture \ref{conj}  for the case of  two-blocks cycles.  Recently, this bound was  improved by  Kim et al.\ to $12k^2$, where $k=\max \{k_1,k_2\}$. \\

For the case of four-blocks cycles $C(k_1,k_2,k_3,k_4)$, the following results are obtained:
\begin{theorem}(Cohen et al.\ \cite{hjk})
	Let $D$ be a strongly connected digraph with no subdivisions of  $C(1,1,1,1)$, then the chromatic number of $D$ is at most $24$.
\end{theorem}

\begin{theorem} (Al-Mniny \cite{4blocks}) Let $k_1$ be a positive integer and let $D$ be a strongly connected digraph  with no subdivisions of $C(k_1,1,1,1)$, then the chromatic number of $D$ is at most $8^{3}.k_1$.
\end{theorem}
\begin{theorem} (Al-Mniny and Zayat \cite{zayat})
 		Let $D$ be a strongly connected digraph having no subdivisions of  $C(k_1,1,k_3,1)$ and let  $k= \textrm{max} \{k_1, k_3\}$, then the chromatic number of $D$ is at most $36. (2k) .(4k+2)$.
 \end{theorem}
 
\noindent For the case of six-blocks cycles, Conjecture \ref{conj} is still unresolved. Also no special cases of six-blocks cycles are studied so far. In this paper, six-blocks cycles are treated for the first time.  We confirm  in this paper Conjecture \ref{conj} for the six-blocks cycles $C(k,1,1,1,1,1)$, namely $g(k,1,1,1,1,1)$ $=O(k)$.  \\

 The paper is organized as follows: In Section 2, we introduce some notations and terminologies that will be used throughout the coming sections.  In Section 3, we prove the existence of subdivisions of $C(k,1,1,1,1,1)$ in strong digraphs by using the technique of digraphs decomposing and the simple notion of a  maximal-tree.  
 
\section{Preliminaries}
Throughout this section, some basic definitions and terminologies that will be used in the coming sections are introduced.\\
 
  A graph $G$ is called $d$-degenerate, if every  subgraph of $G$ has a vertex with at most $d$ neighbors. By inductive techniques, one may easily show that the following statement holds:
\begin{lemma}\label{degenerate}
	If $G$ is $d$-degenerate graph, then $G$ is $(d+1)$-colorable.
\end{lemma} 

Let  $D_{1}$ and $D_{2}$ be two digraph. We define $D_{1} \cup D_{2}$ to be the digraph whose set of vertices is $V(D_{1}) \cup V(D_{2})$ and whose set of arcs is $A(D_{1}) \cup A(D_{2})$. The next well known lemma will be useful for the coming section, as it plays a central role in giving an upper bound for $ \chi(D)$: We partition the set of arcs of $D$ into $A_1, A_2, \dots A_t$, then we give an upper bound for $ \chi(D_i)$, where $D_i$ is the spanning subdigraph of $D$ with arc set $A_i$, for all $i \in \lbrace 1, \dots , t \rbrace$.
 \begin{lemma}\label{far}
 	$\chi(D_{1} \cup D_{2}) \leq \chi(D_{1}) \times \chi(D_{2})$ for any two digraphs  $D_{1}$ and $D_{2}$.
 \end{lemma}

\noindent  Let $P=x_1x_2 \dots x_t$ be a directed path of a digraph $D$. We denote by $P[x_i,x_j]$, $P[x_i,x_j[$, $P]x_i,x_j]$ and $P]x_i,x_j[$, the directed subpath of $P$ from $x_i$ to $x_j$, from $x_i$ to $x_{j-1}$, from $x_{i+1}$ to $x_j$ and from $x_{i+1}$ to $x_{j-1}$ respectively. Similarly this definition is extended to  cycles. Let $x$ be a vertex of a digraph $D$, we denote by ${N_D}^{+}(x)$ (resp. $N_D^{-}(x)$)  the set of vertices $y$ such that $(x,y)$ (resp. $(y,x)$) is an arc of $D$. The cardinality of $N^{+}_D(x)$ (resp. $N^{-}_D(x)$)  is denoted by $d^{+}_D(x)$ (resp. $d^{-}_D(x)$), and is called the out-degree (resp. in-degree) of $x$. Define $\Delta^{+}(D)= \textrm{max}_{x \in V(D)} d^{+}(x)$ to be the maximum out-degree of $D$. Let $C$ be an oriented cycle in a digraph $D$, a vertex $x$ of $C$ is said to be a source if $d^{-}_C(x)=0$, and a vertex $y$ is said to be a sink if $d^{+}_C(x)=0$. For a vertex $x$ of a graph $G$, the set of all neighbors of $x$ in $G$ is denoted by $N_G(x)$,  the cardinality of $N_G(x)$ is denoted by $d_G(x)$, and we define $\delta(G):= \textrm{min}_{x \in V(G)} d_G(x)$.  \\

We define an out-tree to be an oriented digraph containing no cycles, where all vertices have in-degree at most 1. So, an out-tree has a unique vertex of in-degree $0$, called the source. Let $D$ be a digraph with a spanning  out-tree $T$, let $x$ be a vertex of $T$, and let $r$ be the  source of $T$.   Denoted by $l_{T}(x)$, the number of vertices of the unique $rx$-directed path in $T$. We call $l_{T}(x)$ the level of $x$  with respect to $T$. Define $L_{i}(T):=\{x \in V(T) | l_{T}(x)=i\}$, for any positive integer $i$. The ancestors of $x$ are the vertices that belong to $T[r,x]$. If $y$ is an ancestor of $x$ with respect to $T$, we write $ y \leqslant_{T} x$. Denote by  $T_{x}$ the subtree of $T$ rooted at $x$ and induced by the set of vertices $y$ of $D$ such that  $x$ is an ancestor of $y$.  For two vertices $x$ and $y$ of $D$, the least common ancestor $z$ of $x$ and $y$, abbreviated by $\textrm{l.c.a}(x,y)$,   is the common ancestor of $x$ and $y$  with the highest level in $T$. As $r$ is a common ancestor of all vertices, then the latter notion is well-defined.  An arc $(x,y)$ of $D$ is said to be forward with respect to $T$ if $l_{T}(x) < l_{T}(y)$. Otherwise,  $(x,y)$ is  called a backward arc. We say that $T$ is  a final out-tree of $D$, if for every backward arc $(x,y)$ of $D$,  $y \leqslant_{T} x$. Thus, whenever $T$ is final out-tree of $D$,   one may easily see that $D[L_i(T)]$ is an  empty digraph  for all $i \geq 1$. \medbreak 
\noindent 
The following proposition shows an interesting structural property on  digraphs having a spanning out-tree:
\begin{proposition}\label{finaltree}
Let $D$ be a digraph containing a  spanning out-tree $T$, then $D$ contains a final out-tree.
\end{proposition}
\begin{proof} 
Let $T_{0}:=T$. If $T_{0}$ is final, then we are done. Otherwise, $D$ contains two vertices $x$ and $y$, such that $(x,y)$ is a backward arc with respect to $T$ and $y$ is not an ancestor of $x$.  Let $T_{1}$ be the out-tree obtained from $T_{0}$ by removing the arc of head $y$ in $T_{0}$ and adding  $(x,y)$  to $T_{0}$. Clearly, $l_{T_1}(v)\geq l_{T_0}(v)$ for every $v\in V(D)$, and there exists a vertex $y$ such that $l_{T_1}(y)> l_{T_0}(y)$. Now as the level of a vertex cannot increase infinitely, then after a finite number of repeating the above process,  we obtain an out-tree  which is final.$\hfill {\square}$
\end{proof}

 \section{The existence of $S$-$C(k,1,1,1,1,1)$ in strong digraphs}
  
 In this section we found an upper bound of the chromatic number of strongly connected digraphs with no subdivisions of $C(k,1,1,1,1,1)$.  Let  $D$ to be a digraph with no subdivisions of $C(k,1,1,1,1,1)$ and let  $T$ be a final spanning out-tree of $D$ rooted at $r$. Consider the partition of the vertex-set of $D$: $V_1, V_2, \dots,V_{k}$, where   $V_{i}:=\cup_{\alpha \geq 0} L_{i+\alpha (k)}(T)$ for all  $1 \leq i \leq k$.  Now denote by $D_{i}$  the subdigraph of $D$ induced by $V_{i}$, and  partition the arc-set  of $D_{i}$ as follows:$$A_{1}:=\{(x,y)  | l_{T}(x) < l_{T}(y) \hspace{1mm} \textrm{and} \hspace{1mm} x\leqslant_{T}y\};$$ $$A_{2}:=\{(x,y)  | l_{T}(x) > l_{T}(y) \hspace{1mm} \textrm{and} \hspace{1mm} y \leqslant_{T}x\};$$
 $$A_{3}:=A(D_{i}) \setminus (A_{1} \cup A_{2}).$$
 In what follows, denote by  $D_{i}^{j}$   the spanning subdigraph of $D_{i}$ whose arc-set is $A_{j}$, for $1 \leq i \leq k$ and $j=1,2,3$. \vspace{2mm}\\
 The following theorem is very crucial in our later analysis:
 	\begin{theorem}(Cohen et al.\ \cite{hjk})\label{antidi}
	Let $D$ be an oriented graph and $k$ be an integer greater than $1$. If $\chi(D) \geqslant 8k-7$, then $D$ contains an antidirected cycle  of length at least $2k$.
\end{theorem}

\subsection{Coloring $D_i^1$ and $D_i^2$}
 	\begin{proposition}\label{hello}
 	$\chi(D_{i}^{1}) \leqslant 24$ for all $i \in \{1, \dots, k\}$.
 \end{proposition}
 \begin{proof} Assume to the contrary  that $\chi(D_{i}^{1}) \geqslant 25$. Then Theorem \ref{antidi} implies that $D_{i}^{1}$ contains an antidirected cycle $C$ of length at least $8$. Let $x$ be a source of $C$ such that $l_{T}(x)$ is maximal. Let $y$ and $z$ be the out-neighbors of $x$ in $C$. As $C$ is an antidirected cycle, let $s$ and $t$ be the second inneighbor of $y$ and $z$ in $C$ respectively. The definition of $T$ and $D_i^1$ implies that $s$ and $t$ belongs to $T[r,x[$, where $r$ is the root of $T$. Assume without loss of generality that $l_{T}(t)<l_{T}(s)<l_{T}(x)$. For simplicity, let us enumerate $C$ as follows: $C:= x_{0}x_{1}\dots x_{m}$, with $x_0=x$, $x_1=z$, $x_2=t$, $x_m=y$, and $x_{m-1}=s$. Now since $C$ is a cycle (that is it must be closed), $l_{T}(x_2)< l_{T}(x_{m-1})$, and $x_2$ and $x_{m-1}$ are sources in $C$, then the definition of $T$ and $D_i^1$ implies that there exists $i\geq 2$ such that $x_i\in T[r,x_{m-1}[$, $x_{i+1}\in T_{x_{m-1}}$, $x_{i+2}\in \displaystyle{\bigcup_{h\in T[x_{m-1},x_0[}}T_h$, with $(x_i,x_{i+1})\in A(C)$ and $(x_{i+2},x_{i+1})\in A(C)$. Since otherwise, $C$ can't be closed, and this contradicts the definition of a cycle.   Since $x_{i+2}$ is  a source of $C$, then  the maximality of $l_{T}(x_0)$ implies that $l_{T}(x_{i+2})< l_{T}(x_{0})$. This is the reason behind concluding that $x_{i+2}\in \displaystyle{\bigcup_{h\in T[x_{m-1},x_0[}}T_h$. Clearly, since $x_{i+2}\in T_{x_{m-1}}$, then $x_{m-1}$ and $x_{i+2}$ are ancestors. We consider two cases:
 \begin{itemize}
 \item  If $x_i=x_2$, then since $C$ has length at least $8$, we have: $x_{4}\neq x_{m-1}$. This implies that $l(T[x_{m-1},x_{4}])\geq k$.  So the union of $T[x_{m-1},x_{4}]\cup (x_{4},x_{3})$, $(x_2,x_3)$, $(x_2,x_1)$, $(x_0,x_1)$, $(x_0,x_m)$, and $(x_{m-1},x_m)$ is a $S$-$C(k,1,1,1,1,1)$ in $D$, a contradiction.  
 \item Now if $i>2$, then either $x_2\leqslant_{T}x_i$ or $x_i\leqslant_{T}x_2$. Assume first that the former holds, then the union of $T[x_{2},x_{i}]\cup (x_{i},x_{i+1})$, $T[x_{m-1},x_{i+2}]\cup (x_{i+2},x_{i+1})$, $(x_{m-1},x_m)$, $(x_0,x_m)$, $(x_0,x_1)$, and $(x_2,x_1)$ is a $S$-$C(k,1,1,1,1,1)$ in $D$, a contradiction. Then the latter holds, and so the union of $T[x_{i},x_{2}]\cup (x_{2},x_{1})$, $(x_{0},x_1)$, $(x_0,x_m)$, $(x_{m-1},x_m)$, $T[x_{m-1},x_{i+2}]\cup (x_{i+2},x_{i+1})$ and $(x_i,x_{i+1})$ is a $S$-$C(k,1,1,1,1,1)$ in $D$, a contradiction.
 \end{itemize} This completes the proof. $\hfill {\square}$
\end{proof} 
 
 \begin{proposition}\label{prop2}
 	$\chi(D_{i}^{2}) \leqslant 7$ for all $i \in \{1,\dots , k\}$.
 \end{proposition}
 \begin{proof}
Partition $V(D_i^2)$ into two sets, $S_1:= \lbrace x \in V(D_i^2)$, such that $d_{D_i^2}^+(x) \leq 1 \rbrace$, and $S_2:= V(D_i^2) \backslash S_1 $. Now, we show that $\Delta^+ (D_i^2[S_2]) \leq 4$.  Let $x \in S_2$. By the definition of $D_i^2$, all the out neighbors of $x$ in $D_i^2$ are ancestors. Set $N^+_{D_i^2}(x)= \lbrace y_1,y_2, \dots, y_p \rbrace$, and without loss of generality, assume that $y_i \leqslant_T y_{i+1}$, for all $i \in \lbrace 1, 2, \dots , p-1 \rbrace$. Assume to the contrary that there exist at least $3$ vertices in $ \lbrace y_2, y_3 , \dots, y_{p-1} \rbrace$ that belong to $S_2$. Assume without loss of generality that $y_2$, $y_3$ and $y_4$ belong to these vertices. Let $z_2$ be an out-neighbor of $y_2$ in $D_i^2$. As $d^+_{D_i^2}(y_4) \geq 2$, then $y_4$ has an out neighbor $z_4 \in D_i^2$, such that $z_4 \neq y_2$. If $z_4 \in T_{y_2} \backslash \lbrace y_2 \rbrace$, then the union of  $T[y_2,z_4]$, $(y_4,z_4)$, $T[y_4,y_p]$, $(x,y_p)$, $(x,y_1) \cup T[y_1,z_2]$ and $(y_2,z_2)$ is a $S$-$C(k,1,1,1,1,1)$ whenever $z_2 \in T_{y_1}$, and the union of $T[y_2,z_4]$, $(y_4,z_4)$, $T[y_4,y_p]$, $(x,y_p)$, $(x,y_1)$ and $(y_2,z_2) \cup T[y_1,z_2] $  is a $S$-$C(k,1,1,1,1,1)$ whenever $z_2 \leqslant_T y_1$, a contradiction. So, $z_4 \leqslant_T y_2$. Then the union of $T[y_4,y_p]$, $(x,y_p)$, $(x,y_3)$, $T[y_2,y_3]$, $(y_2,z_2) \cup T[z_2,z_4]$  and $(y_4,z_4)$ is a $S$-$C(k,1,1,1,1,1)$ whenever $z_2 \leqslant_T x_4$, and the union of $T[y_4,y_p]$, $(x,y_p)$, $(x,y_3)$, $T[y_2,y_3]$, $(y_2,z_2)$ and $(y_4,z_4) \cup T[z_4,z_2]$ is a $S$-$C(k,1,1,1,1,1)$ whenever $z_4 \leqslant_T z_2$,  a contradiction. Hence $x$ has at most $2$ out neighbors in  $ \lbrace y_2, y_3 , \dots, y_{p-1} \rbrace$ that belong to $S_2$. Thus $x$ has at most $4$ out neighbors in $S_2$. Note that every acyclic digraph is $d$-degenerate, where $d$ is its maximal outdegree. As $D_i^2$ is acyclic, then $D_i^2[S_1]$ and $D_i^2[S_2]$ are acyclic. So Lemm \ref{degenerate} implies that  $\chi(D_i^2[S_1]) \leq \Delta^+(D_i^2[S_1]) +1 \leq 2$ and $\chi(D_i^2[S_2]) \leq \Delta^+(D_i^2[S_2]) +1 \leq 5$.  Therefore $\chi(D_i^2) \leq \chi(D_i^2[S_1]) + \chi(D_i^2[S_2]) \leq 7$.   This terminates the proof of  Proposition \ref{prop2}.$\hfill {\square}$
\end{proof}

 \subsection{The landscape of antidirected cycles in $D_{i}^{3}$ and coloring $D_i^3$}
 \begin{figure}[h]
	\centering
	\includegraphics[width=0.9\linewidth]{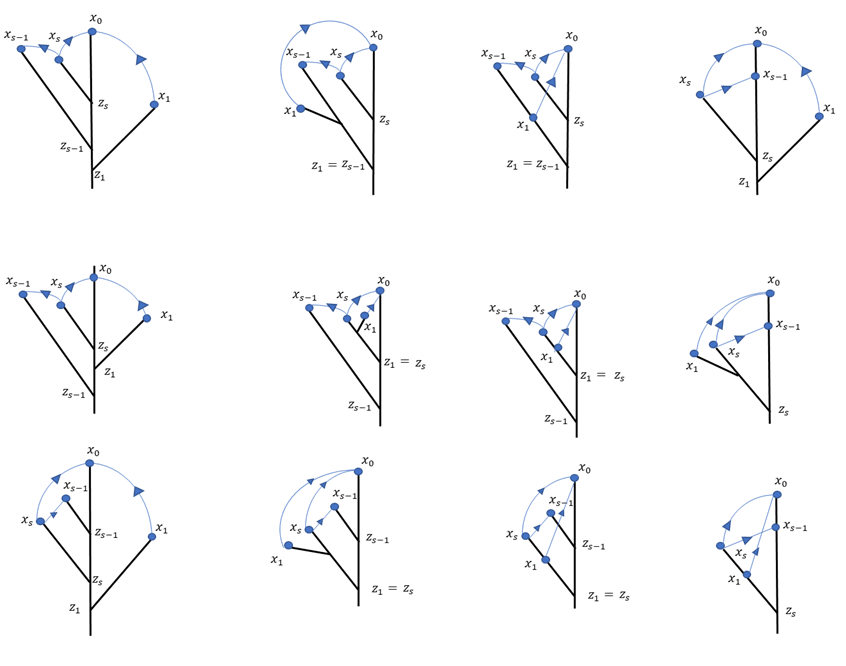}
	\caption{All possible positions of $z_1, z_{s-1}$, and $z_s$ on the tree $T$.}
	\label{fig:cases}
\end{figure} \vspace{3mm}

 In this section, we are going to show that $D_i^3$ contains no antidirected cycle of length at least $8$. To this end, we assume that the contrary holds. Let  $C:=x_0x_1 \dots x_s$ be an antidirected cycle  of length at least $8$ in $D_i^3$, where $x_0$ is a sink of $C$ of maximal level among the vertices of $C$, and let $z_i=\textrm{l.c.a}(x_i,x_0)$, for all $ i \in \lbrace 1,2, \dots ,s \rbrace $. We have $z_i \leqslant_T x_0$, for all $i \in \lbrace 1, \dots , s \rbrace $. Without loss of generality, assume that $z_1 \leqslant_T z_s$, and if $z_1= z_s$, then assume that $l_T(x_1) \leq l_T(x_s)$. We have  $z_i \leqslant_T z_j \leqslant_T z_k $, where $i$, $j$ and $k$ $\in \lbrace 1, s-1, s \rbrace$ (see Figure \ref{fig:cases}). In what follows, we study the case where $z_1 \leqslant_T z_{s-1} \leqslant_T z_s$, such that $z_1$, $z_{s-1}$ and $z_s$ are pairwisely distinct, and also we study the case where $z_1 = z_{s} $, such that $x_1 \leqslant_T x_s$ and $x_{s-1} \leqslant_T x_0$. The other cases are done analogously and we omit them.\\
In each of the studied cases, we proceed by studying the position of $x_2$ in order to find a $S$-$C(k,1,1,1,1,1)$ in $D$ and reach a contradiction, which confirms that  $D_i^3$ contains no antidirected cycle of length at least $8$ and so by Theorem \ref{antidi}, $ \chi (D_i^3) \leq 24$.

In what follows, we consider the case:  $z_1 \leqslant_T z_{s-1} \leqslant_T z_s $, where $z_1$, $z_{s-1}$ and $z_s$ are pairwisely distinct. We show first that  $x_2$ and  $x_{s-1}$ are not ancestors: In Lemma \ref{lem1} (respectively Lemma \ref{lem2}) (respectively Lemma \ref{lem3}), we prove that if $x_2 \in T_{x_{s-1}} $ (respectively $x_2 \in T]z_{s-1},x_{s-1}[ $) (respectively $x_2 \in T]z_1,z_{s-1}]$), then $D$ contains a $S$-$C(k,1,1,1,1,1)$.\\
\begin{figure}[h]
	\centering
	\includegraphics[width=0.75\linewidth]{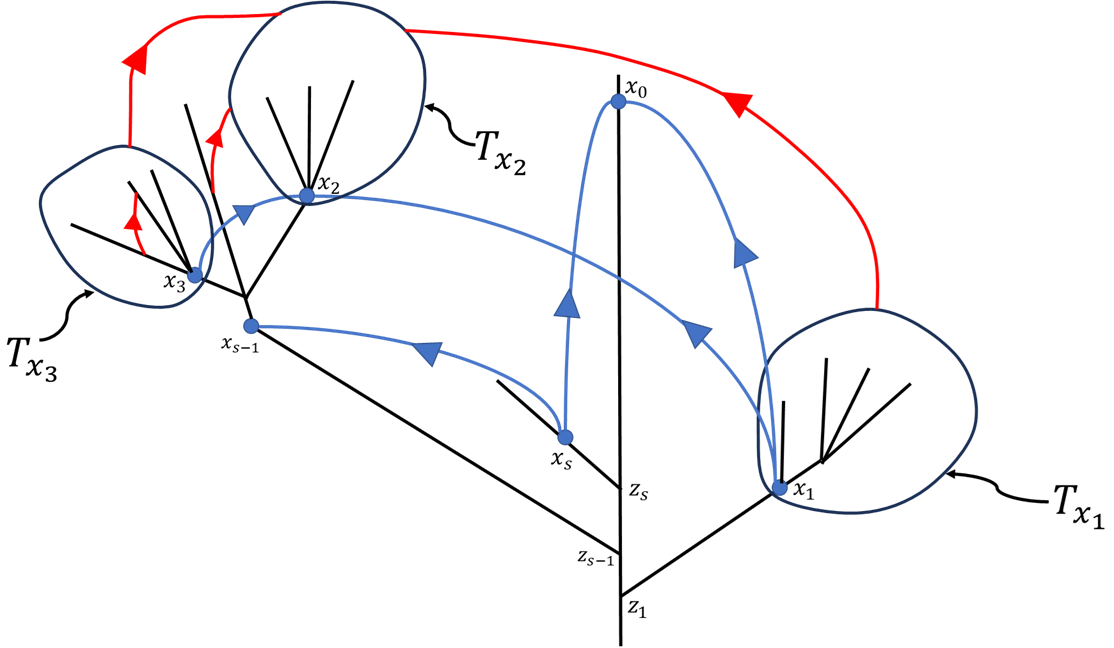}
	\caption{Illustration of Assertion \ref{assa} and Assertion \ref{ass1}.}
	\label{fig:possible arcs1}
\end{figure} \vspace{3mm}
 
 \begin{lemma}\label{lem1}
 Suppose that $D_i^3$ contains an antidirected cycle $C:=x_0x_1 \dots x_s$ of length at least $6$, where $x_0$ is a sink of $C$ of maximal level among the vertices of $C$, and let $z_i=\mathrm{l.c.a}(x_i,x_0)$, for all $ i \in \lbrace 1,2, \dots ,s \rbrace $, such that $z_1 \leqslant_T z_{s-1} \leqslant_T z_s $, where $z_1$, $z_{s-1}$ and $z_s$ are pairwisely distinct. Then $x_2 \notin T_{x_{s-1}} $.
\end{lemma}
\begin{proof} Assume to the contrary that $x_2 \in T_{x_{s-1}}$.
 \begin{assertion}\label{assa}
 For all $ y \in T_{x_2}$, $N^-_{D_i^3}(y) \subseteq T_{x_1} \cup T_{x_{s-1}}$.
 \end{assertion}
 \begin{proof} 
  Assume otherwise. Suppose that there exists $ y \in T_{x_2}$, such that $(x,y) \in A({D_i^3})$ and $x \notin T_{x_1} \cup T_{x_{s-1}}$. As $ x\notin T_{x_{s-1}}$, then $x$ and $x_{s-1}$ are not ancestors. Let $z=\textrm{l.c.a}(x,x_{s-1})$. If $z \leqslant _T z_{s-1}$, then the union of $T[z,x_{s-1}]$, $(x_s,x_{s-1})$, $(x_s,x_0)$, $(x_1,x_0)$, $(x_1,x_2) \cup T[x_2,y]$ and $T[z,x] \cup (x,y)$ is a $S$-$C(k,1,1,1,1,1)$ whenever $x \notin T_{x_s}$, and the union of $T[x_s,x] \cup (x,y)$, $(x_1,x_2) \cup T[x_2,y]$,  $(x_1,x_0)$, $T[z_{s-1},x_0]$, $T[z_{s-1},x_{s-1}]$ and $(x_s,x_{s-1})$ is a $S$-$C(k,1,1,1,1,1)$ whenever $x \in T_{x_s}$, a contradiction. If $z_{s-1} \leqslant_T z$ and $ z_{s-1} \neq z$, then the union of $T[z_s,x_0]$, $(x_1,x_0)$, $(x_1,x_2) \cup  T[x_2,y]$, $T[z,x] \cup (x,y)$, $T[z,x_{s-1}]$ and $T[z_s,x_s] \cup (x_s,x_{s-1})$ is a $S$-$C(k,1,1,1,1,1)$, a contradiction. This terminates the proof of this assertion. $\hfill {\lozenge}$
  \end{proof}\vspace{3mm}\\
  Due to Assertion \ref{assa}, $x_3 \in T_{x_1} \cup T_{x_{s-1}}$. Assume first that $x_3 \in T_{x_{s-1}}$.
\begin{assertion}\label{ass1} 
 For all $x\in T_{x_3}$, $N^{+}_{D_i^3}(x) \subseteq T_{x_2} \cup T_{x_3} $.
 \end{assertion}
 \begin{proof}
 Suppose to the contrary that there exists $ x\in T_{x_3}$, such that $(x,y) \in A({D_i^3})$ and $y \notin T_{x_2} \cup T_{x_3}$. If $y \leqslant_T x_0$, then the union of $(x_s,x_{s-1})\cup T[x_{s-1},x_2]$, $(x_3,x_2)$, $T[x_3,x] \cup (x,y)$, $T[z_1,y]$, $T[z_1,x_1] \cup (x_1,x_0)$ and $(x_s,x_0)$ is a $S$-$C(k,1,1,1,1,1)$, a contradiction. So $y$ is not an ancestor of $x_0$. Let $z=\textrm{l.c.a}(x_0,y)$. Then the union of  $(x_s,x_{s-1})\cup T[x_{s-1},x_2]$, $(x_3,x_2)$, $T[x_3,x] \cup (x,y)$, $T[z,y]$, $T[z,x_0]$ and $(x_s,x_0)$ is a $S$-$C(k,1,1,1,1,1)$ whenever $y\notin T_{x_s} \cup T_{x_{s-1}}$, and the union of $T[x_s,y]$, $T[x_3,x]\cup (x,y)$, $(x_3,x_2)$, $(x_1,x_2)$, $(x_1,x_0)$ and $(x_s,x_0)$ is a $S$-$C(k,1,1,1,1,1)$ whenever $y \in T_{x_s}$, and $(x_s,x_{s-1}) \cup T[x_{s-1},y]$, $T[x_3,x]\cup(x,y)$, $(x_3,x_2)$, $(x_1,x_2)$, $(x_1,x_0)$ and $(x_s,x_0)$ is a $S$-$C(k,1,1,1,1,1)$ whenever $y \in T_{x_{s-1}} \backslash ( T_{x_2} \cup T_{x_3})$, a contradiction.   This implies the desired result.  $\hfill {\lozenge}$
\end{proof} \vspace{3mm}\\
Due to Assertion \ref{ass1} and the definition of $D_i^3$, we have $x_4 \in T_{x_2}$.

\begin{assertion}\label{newassertion1}
 For all $ x \in T_{x_{s-1}}$, $N^{+}_{D_i^3}(x) \subseteq T_{x_{s-1}} \cup T_{x_1}$.
 \end{assertion} 
  \begin{proof}
Assume to the contrary that there exists an arc $(x,y) \in A({D_i^3})$, such that $x \in T_{x_{s-1}}$  and $y \notin T_{x_1} \cup T_{x_{s-1}}$. We proceed according to the position of $x$.\\
First, suppose that $x \in T_{x_2}$. Due to the definition of $D_i^3$, $y$ and $x_{s-1}$ are not ancestors. Let $z=\textrm{l.c.a}(x_{s-1},y)$, then union of $T[z,y]$, $(x_1,x_2) \cup T[x_2,x] \cup (x,y)$, $(x_1,x_0)$, $(x_s,x_0)$, $(x_s,x_{s-1})$ and $T[z,x_{s-1}]$ forms a $S$-$C(k,1,1,1,1,1)$ whenever $y \notin T_{x_s}$, and $T[x_s,y]$, $(x_1,x_2) \cup T[x_2,x] \cup (x,y)$, $(x_1,x_0)$, $T[z_{s-1},x_0]$, $T[z_{s-1},x_{s-1}]$ and $(x_s,x_{s-1})$ form a $S$-$C(k,1,1,1,1,1)$ whenever $y \in T_{x_s}$, a contradiction.  So, $x \notin T_{x_2}$.\\
Now, suppose that $x \in T]x_{s-1},x_2[$. Then, y is not ancestor of $x_0$, since elsewhere, the union of $T[z_s,y]$, $(x,y)$, $T[x,x_2]$, $(x_1,x_2)$, $(x_1,x_0)$ and $T[z_s,x_s] \cup (x_s,x_0)$ forms a $S$-$C(k,1,1,1,1,1)$, a contradiction. Let $z=\textrm{l.c.a}(x_0,y)$, then $T[x,x_2]$, $(x_1,x_2)$, $(x_1,x_0)$, $T[z,x_0]$, $T[z,y]$ and $(x,y)$ form a $S$-$C(k,1,1,1,1,1)$, a contradiction.\\
Thus, $x \in T_{x_{s-1}} \backslash (T_{x_2} \cup T[x_{s-1},x_2])$. Let $z'=\textrm{l.c.a}(x_2,x)$. If $y \in T_{x_s}$, then $T[x_s,y]$, $T[z',x] \cup (x,y)$, $T[z',x_2]$, $(x_1,x_2)$, $(x_1,x_0)$ and $(x_s,x_0)$ form a $S$-$C(k,1,1,1,1,1)$, a contradiction. So, $y$ and $x_s$ are not ancestors. Let $z=\textrm{l.c.a}( x_s,y)$, then $T[z,y]$, $T[z',x] \cup (x,y)$, $T[z',x_2]$, $(x_1,x_2)$, $(x_1,x_0)$ and $T[z,x_s] \cup (x_s,x_0)$ form a $S$-$C(k,1,1,1,1,1)$, a contradiction. This terminates the proof of this assertion.$\hfill {\lozenge}$ 
 \end{proof} 
\begin{assertion}\label{newassertion2}
For all $ y \in T_{x_{s-1}}$, $N^{-}_{D_i^3}(y) \subseteq T_{x_{s-1}} \cup T_{x_1}$. 
 \end{assertion}
 \begin{proof} Assume, to the contrary, that there exists an arc $(x,y) \in A(D_i^3)$, such that $y \in T_{x_{s-1}}$ and $x \notin T_{x_1} \cup T_{x_{s-1}}$. By Assertion \ref{assa}, $y \notin T_{x_2}$. If $y \leqslant_T x_2$, then $x \notin T_{x_s}$, otherwise the union of $T[x_s,x] \cup (x,y) \cup T[y,x_2]$, $(x_1,x_2)$, $(x_1,x_0)$,  $T[z_{s-1},x_0]$, $T[z_{s-1},x_{s-1}]$ and $(x_s,x_{s-1})$ forms a $S$-$C(k,1,1,1,1,1)$, a contradiction. Let $z=\textrm{l.c.a}(x_{s-1},x)$, then  $T[z,x] \cup (x,y) \cup T[y,x_2]$, $(x_1,x_2)$, $(x_1,x_0)$, $(x_s,x_0)$, $(x_s,x_{s-1})$ and $T[z,x_{s-1}]$  forms a $S$-$C(k,1,1,1,1,1)$, a contradiction. So, $y$ is not an ancestor of $x_2$.\\
 Let $z'=\textrm{l.c.a}(x_2,y)$, then $x$ is not ancestor of $x_0$, otherwise the union of $T[x,x_0]$, $(x_1,x_0)$, $(x_1,x_2)$, $T[z',x_2]$, $T[z',y]$ and $(x,y)$  forms a $S$-$C(k,1,1,1,1,1)$, a contradiction. Let $z=\textrm{l.c.a}(x_0,x)$, then $T[z,x_0]$, $(x_1,x_0)$, $(x_1,x_2)$, $T[z',x_2]$, $T[z'y]$ and  $T[z,x] \cup (x,y)$  forms a $S$-$C(k,1,1,1,1,1)$, a contradiction.$\hfill {\lozenge}$
 \end{proof}\\
 
 Let $i$ be the smallest integer, such that $x_i \in T_{x_2} \cup T_{x_3}$ and $x_{i+1} \notin T_{x_2} \cup T_{x_3}$,  $i \in \lbrace 4,5 \dots s-2 \rbrace$. Due to Assertion \ref{newassertion1} and Assertion \ref{newassertion2}, we have $x_{i+1} \in T_{x_{s-1}} \cup T_{x_1}$. We discuss according to the position of $x_i$.\\

 $\textbf{Case 1:}$ If $x_i \in T_{x_3}$ and $x_i$ is a source of $C$.\\
 If $x_{i+1} \in T_{x_{s-1}} \backslash (T_{x_2} \cup T_{x_3})$, then the union of $(x_s,x_{s-1}) \cup T[x_{s-1},x_{i+1}]$, $T[x_3,x_i] \cup (x_i,x_{i+1})$, $(x_3,x_2)$, $(x_1,x_2)$, $(x_1,x_0)$ and $(x_s,x_0)$  forms a $S$-$C(k,1,1,1,1,1)$, a contradiction. So, $x_{i+1} \in T_{x_1}$. Thus, $T[x_1,x_{i+1}]$, $T[x_3,x_i] \cup (x_i,x_{i+1})$, $(x_3,x_2)$, $(x_s,x_{s-1}) \cup T[x_{s-1},x_2]$, $(x_s,x_0)$ and $(x_1,x_0)$  form a $S$-$C(k,1,1,1,1,1)$, a contradiction. So, this case is rejected.\\
 
 $\textbf{Case 2:}$ If $x_i \in T_{x_3}$ and $x_i$ is a sink of $C$.\\
 If $x_{i+1} \in T_{x_{s-1}} \backslash ( T_{x_2} \cup T_{x_3})$, then the union of $(x_s,x_{s-1}) \cup T[x_{s-1},x_{i+1}] \cup (x_{i+1},x_i)$, $T[x_3,x_i]$, $(x_3,x_2)$, $(x_1,x_2)$, $(x_1,x_0)$ and $(x_s,x_0)$ forms a $S$-$C(k,1,1,1,1,1)$, a contradiction. So, $x_{i+1} \in T_{x_1}$. Thus, $(x_s,x_{s-1}) \cup T[x_{s-1},x_2]$, $(x_3,x_2)$, $T[x_3,x_i]$, $T[x_1,x_{i+1}] \cup (x_{i+1},x_i)$, $(x_1,x_0)$ and $(x_s,x_0)$ form a $S$-$C(k,1,1,1,1,1)$, a contradiction.\\
 
 $\textbf{Case 3:}$ If $x_i \in T_{x_2}$ and $x_i$ is a source of $C$.\\
 Due to the minimality of $i$, $x_{i-1} \in T_{x_2} \cup  T_{x_3}$. Assume first that $x_{i-1} \in T_{x_2}$, then the union of $(x_1,x_2) \cup T[x_2,x_{i-1}]$, $(x_i,x_{i-1})$, $(x_i,x_{i+1})$, $(x_s,x_{s-1}) \cup T[x_{s-1},x_{i+1}]$, $(x_s,x_0)$ and $(x_1,x_0)$ forms a $S$-$C(k,1,1,1,1,1)$ whenever $x_{i+1} \in T_{x_{s-1}} \backslash (T_{x_2} \cup T_{x_3})$, and $T[x_1,x_{i+1}]$, $(x_i,x_{i+1})$, $(x_i,x_{i-1})$, $(x_s,x_{s-1}) \cup T[x_{s-1},x_{i-1}]$, $(x_s,x_0)$ and $(x_1,x_0)$ form a $S$-$C(k,1,1,1,1,1)$ whenever $x_{i+1} \in T_{x_1}$, a contradiction.\\
 Now, assume that $x_{i-1} \in T_{x_3}$, then the union of $T[x_3,x_{i-1}]$, $(x_i,x_{i-1})$, $(x_i,x_{i+1})$, $T[z_1,x_{i+1}]$, $T[z_1,x_1] \cup (x_1,x_2)$ and $(x_3,x_2)$    forms a $S$-$C(k,1,1,1,1,1)$ whenever $x_{i+1} \in T_{x_{s-1}} \backslash (T_{x_2} \cup T_{x_3})$, and $T[x_1,x_{i+1}]$, $(x_i,x_{i+1})$, $(x_i,x_{i-1})$, $(x_s,x_{s-1}) \cup T[x_{s-1}, x_{i-1}]$, $(x_s,x_0)$ and $(x_1,x_0)$ form a $S$-$C(k,1,1,1,1,1)$ whenever $x_{i+1} \in T_{x_1}$, a contradiction.\\
 
  $\textbf{Case 4:}$ If $x_i \in T_{x_2}$ and $x_i$ is a sink of $C$.\\
   Due to the minimality of $i$, $x_{i-1} \in T_{x_2} \cup  T_{x_3}$ and $x_{i-2} \in T_{x_2} \cup T_{x_3}$. If $x_{i-1} \in T_{x_3}$, then the union of $(x_s,x_{s-1}) \cup T[x_{s-1},x_{i+1}] \cup (x_{i+1},x_i)$, $T[x_3,x_{i-1}] \cup(x_{i-1},x_i)$, $(x_3,x_2)$, $(x_1,x_2)$, $(x_1,x_0)$ and $(x_s,x_0)$ forms a $S$-$C(k,1,1,1,1,1)$ whenever $x_{i+1} \in T_{x_{s-1}} \backslash (T_{x_2} \cup T_{x_3})$, and $T[x_1,x_{i+1}] \cup (x_{i+1},x_i)$, $T[x_3,x_{i-1}] \cup(x_{i-1},x_i)$, $(x_3,x_2)$, $(x_s,x_{s-1}) \cup  T[x_{s-1},x_2]$, $(x_s,x_0)$ and $(x_1,x_0)$   form a $S$-$C(k,1,1,1,1,1)$ whenever $x_{i+1} \in T_{x_1}$, a contradiction. So, $x_{i-1} \in T_{x_2}$.\\
   If $x_{i+1} \in T_{x_1}$, then $T[x_1,x_{i+1}] \cup (x_{i+1},x_i)$, $T[z,x_{i-1}] \cup (x_{i-1},x_i)$, $T[z,x_4]$, $(x_s,x_{s-1}) \cup T[x_{s-1},x_3] \cup (x_3,x_4)$, $(x_s,x_0)$ and $(x_1,x_0)$ form a $S$-$C(k,1,1,1,1,1)$ whenever $x_{i-1} \in T_{x_2} \backslash T_{x_4}$ where $z= \textrm{l.c.a}(x_4,x_{i-1})$, and $T[x_1,x_{i+1}] \cup (x_{i+1},x_i)$, $(x_{i-1},x_i)$, $T[x_{i-1},x_4]$, $(x_3,x_4)$, $(x_3,x_2)$ and $(x_1,x_2)$ form a $S$-$C(k,1,1,1,1,1)$ whenever $x_{i-1} \in T]x_2,x_4[$, and $T[x_1,x_{i+1}] \cup (x_i,x_{i+1})$, $(x_3,x_4) \cup T[x_4,x_{i-1}] \cup (x_{i-1},x_i)$, $(x_3,x_2)$, $(x_s,x_{s-1}) \cup T[x_{s-1},x_2]$, $(x_s,x_0)$ and $(x_1,x_0)$ form a $S$-$C(k,1,1,1,1,1)$ whenever $x_{i-1} \in T_{x_4}$, a contradiction. \\
   So, $x_{i+1} \in T_{x_{s-1}} \backslash (T_{x_2} \cup T_{x_3})$. If $x_{i-2} \in T_{x_2}$, then the union of $(x_1,x_2) \cup T[x_2,x_{i-2}]$, $(x_{i-1},x_{i-2})$, $(x_{i-1},x_i)$, $(x_s,x_{s-1}) \cup T[x_{s-1},x_{i+1}] \cup (x_{i+1},x_i)$, $(x_s,x_0)$ and $(x_1,x_0)$ forms a $S$-$C(k,1,1,1,1,1)$, a contradiction. If $x_{i-2} \in T_{x_3}$, then the union of $(x_s,x_{s-1}) \cup T[x_{s-1},x_{i-2}]$, $(x_{i-1},x_{i-2})$, $(x_{i-1},x_i)$, $(x_1,x_2) \cup T[x_2,x_i]$, $(x_1,x_0)$ and $(x_s,x_0)$ form a $S$-$C(k,1,1,1,1,1)$, a contradiction.\\

    Therefore, by the previous assertions, for all  $x \in V(C) \backslash \lbrace x_0,x_1,x_s,x_{s-1} \rbrace $, $x \in T_{x_2} \cup T_{x_3}$. Particularly, $x_{s-2} \in T_{x_2} \cup T_{x_3} \subseteq T_{x_{s-1}}$ , which contradicts the definition of $D_i^3$. Hence $x_3 \in T_{x_1}$. And by similar proof to the previous arguments, we obtain that $x_i \in T_{x_2} \cup T_{x_3}$, for all $i \in \lbrace 2,3,\dots,s-2 \rbrace $.  And due to the definition of $D_i^3$, we get that $x_{s-2} \in T_{x_3}$. In this case, the union of $T[x_3,x_{s-2}] \cup (x_{s-2},x_{s-1})$, $(x_{s},x_{s-1})$, $(x_s,x_0)$, $(x_1,x_0)$, $(x_1,x_2)$ and $(x_3,x_2)$ is a $S$-$C(k,1,1,1,1,1)$, a contradiction.  This terminates the proof of Lemma \ref{lem1}.$\hfill {\square}$
    \end{proof}
   
    \begin{lemma}\label{lem2}
 Suppose that $D_i^3$ contains an antidirected cycle $C:=x_0x_1 \dots x_s$ of length at least $6$, where $x_0$ is a sink of $C$ of maximal level, and let $z_i=\mathrm{l.c.a}(x_i,x_0)$, for all $ i \in \lbrace 1,2, \dots ,s \rbrace $, such that $z_1 \leqslant_T z_{s-1} \leqslant_T z_s $ and $z_1$, $z_{s-1}$ and $z_s$ are pairwisely distinct. Then $x_2 \notin T]z_{s-1},x_{s-1}[ $. 
\end{lemma}
    \begin{proof}
   Assume that the contrary holds. By similar proof to Lemma \ref{lem1}, we get that if  $x_{s-2} \in T_{x_2}$, then $V(C) \backslash \lbrace x_0,x_1,x_2,x_s \rbrace \subseteq T_{x_{s-1}} \cup T_{x_{s-2}}$, in particular $x_3 \in T_{x_{s-1}} \cup T_{x_{s-2}} \subseteq T_{x_2}$, which contradicts the definition of $D_i^3$, so $x_{s-2} \notin T_{x_2}$. Similarly, $x_{s-2} \notin T_{x_s}$. Also, we note that $x_{s-2} \notin T_{x_1}$, since if otherwise, the union of $T[x_1,x_{s-2}] \cup (x_{s-2},x_{s-1})$, $(x_s,x_{s-1})$, $(x_s,x_0)$, $T[z_{s-1},x_0]$, $T[z_{s-1},x_2]$ and $(x_1,x_2)$ is a $S$-$C(k,1,1,1,1,1)$, a contradiction. Also, $x_{s-2}$ is not an ancestor of $x_1$, since elsewhere the union of $T[z_1,x_2]$, $(x_1,x_2)$, $(x_1,x_0)$, $(x_s,x_0)$, $(x_s,x_{s-1})$ and $T[z_1,x_{s-2}] \cup (x_{s-2},x_{s-1})$ forms a $S$-$C(k,1,1,1,1,1)$, a contradiction. So, $x_{s-2}$ and $x_1$ are not ancestors.  Let $z=\textrm{l.c.a}(x_1,x_{s-2})$, then the union of $T[z_{s-1},x_0]$, $(x_s,x_0)$, $(x_s,x_{s-1})$, $T[z,x_{s-2}] \cup (x_{s-2},x_{s-1})$, $T[z,x_1] \cup (x_1,x_2)$ and $T[z_{s-1},x_2]$  forms a $S$-$C(k,1,1,1,1,1)$ whenever $x_{s-2} \notin T_{z_{s-1}} $. Thus, $x_{s-2} \in T_{z_{s-1}} \backslash (T_{x_s} \cup T_{x_2})$. Now, we proceed according to the position of $x_{s-3}$. The union of  $(x_s,x_{s-1}) \cup  T[x_{s-1},x_{s-3}]$, $T[z',x_{s-2}] \cup (x_{s-2},x_{s-3})$, $T[z',x_2]$, $(x_1,x_2)$, $(x_1,x_0)$ and $(x_s,x_0)$ is a $S$-$C(k,1,1,1,1,1)$ whenever $x_{s-3} \in T_{x_{s-1}}$, where $z'=\textrm{l.c.a}(x_{s-2},x_2)$, and $T[x_s,x_{s-3}]$, $(x_{s-2},x_{s-3})$, $(x_{s-2},x_{s-1})$, $(x_1,x_2) \cup  T[x_2,x_{s-1}]$, $(x_1,x_0)$ and $(x_s,x_0)$ is a $S$-$C(k,1,1,1,1,1)$ whenever $x_{s-3} \in T_{x_s}$, and $T[x_1,x_{s-3}]$, $(x_{s-2},x_{s-3})$, $(x_{s-2},x_{s-1})$, $(x_s,x_{s-1})$, $(x_s,x_0)$ and $(x_1,x_0)$ forms a $S$-$C(k,1,1,1,1,1)$ whenever $x_{s-3} \in T_{x_1}$, a contradiction. So, $x_{s-3} \notin T_{x_1} \cup T_{x_s} \cup  T_{x_{s-1}}$. Also $x_{s-3}$ is not an ancestor of $x_1$, since elsewhere, the union of $T[z_1,x_0]$, $(x_s,x_0)$, $(x_s,x_{s-1})$, $(x_{s-2},x_{s-1})$, $(x_{s-2},x_{s-3})$ and $T[z_1,x_{s-3}]$ forms a $S$-$C(k,1,1,1,1,1)$ whenever $x_{s-2}$ is not an ancestor of $x_0$, and $T[x_{s-2},x_0]$, $(x_s,x_0)$, $(x_s,x_{s-1})$, $T[z_1,x_{s-1}]$, $T[z_1,x_{s-3}]$ and $(x_{s-2},x_{s-3})$ forms a $S$-$C(k,1,1,1,1,1)$ whenever $x_{s-2} \leqslant_T x_0$, a contradiction. Hence $x_{s-3}$ and $x_1$ are not ancestors. Let $z'=\textrm{l.c.a}(x_{s-3},x_1)$, then $T[z',x_{s-3}] $, $(x_{s-2},x_{s-3})$, $(x_{s-2},x_{s-1})$, $(x_s,x_{s-1})$, $(x_s,x_0)$ and $T[z',x_1]\cup (x_1,x_0)$ forms a $S$-$C(k,1,1,1,1,1)$ whenever $l(T[z',x_{s-3}]) \geq k$, a contradiction. So, $l(T[z',x_{s-3}]) <k$, and $z_1 \leqslant_T z'$ and $z_1 \neq z'$. Then, $(x_1,x_2) \cup T[x_2,x_{s-1}]$, $(x_{s-2},x_{s-1})$, $(x_{s-2},x_{s-3})$, $T[z_1,x_{s-3}]$, $T[z_1,x_0]$ and $(x_1,x_0)$ forms a $S$-$C(k,1,1,1,1,1)$ whenever $x_{s-2}$ is not ancestor of $x_0$, and $T[x_{s-2} ,x_0]$, $(x_s,x_0)$, $(x_s,x_{s-1})$, $T[z',x_1] \cup (x_1,x_2) \cup T[x_2,x_{s-1}]$, $T[z',x_{s-3}]$ and $(x_{s-2},x_{s-3})$ forms a $S$-$C(k,1,1,1,1,1)$ whenever $x_{s-2} \leqslant_T x_0$, a contradiction. 
   Hence $x_2 \notin T]z_{s-1},x_{s-1}[$.\\ 
     This terminates the proof of Lemma \ref{lem2}.$\hfill {\square}$\\
     \end{proof} 
     \begin{lemma}\label{lem3} 
     Suppose that $D_i^3$ contains an antidirected cycle $C:=x_0x_1 \dots x_s$ of length at least $6$,  where $x_0$ is a sink of $C$ of maximal level, and let $z_i=\mathrm{l.c.a}(x_i,x_0)$, for all $ i \in \lbrace 1,2, \dots ,s \rbrace $, such that $z_1 \leqslant_T z_{s-1} \leqslant_T z_s $ and $z_1$, $z_{s-1}$ and $z_s$ are pairwisely distinct. Then $x_2 \notin T]z_1,z_{s-1}]$.
     \end{lemma}
     \begin{proof}
   Assume to the contrary  that $x_2 \in T]z_1,z_{s-1}]$.
  \begin{assertion} There exists no arc $(x,y) \in A(D_i^3)$, such that $x\in T_{z} \backslash (T_{x_1} \cup T_{x_2})$ and $y \in T_{x_2} \backslash \lbrace x_2 \rbrace$, for any $z \leqslant x_2$.
  \end{assertion} 
  \begin{proof}
 Suppose to the contrary that such an arc exists. Let $z=\textrm{l.c.a}(x_0,x)$. If $y$  and $x_{s-1}$ are ancestors, then the union of $T[z,x] \cup (x,y) \cup T[y,x_{s-1}]$, $(x_{s},x_{s-1})$, $(x_s,x_0)$, $(x_1,x_0)$, $(x_1,x_2)$ and $T[z,x_2]$ is a $S$-$C(k,1,1,1,1,1)$ whenever $y \in T]x_2,x_{s-1}[$, and the union of  $T[z_s,x_0]$, $(x_1,x_0)$, $(x_1,x_2)$, $T[z,x_2]$, $T[z,x] \cup (x,y)$ and $T[z_s,x_s] \cup (x_s,x_{s-1}) \cup T[x_{s-1},y]$ is a $S$-$C(k,1,1,1,1,1)$ whenever $y \in T_{x_{s-1}}$, a contradiction. So $y$ and $x_{s-1}$ are not ancestors. Let $z'=\textrm{l.c.a}(y,x_{s-1})$ and $z''=\textrm{l.c.a}(x,x_1)$. If $z' \leqslant _T z_{s-1}$, then the union of $T[z',x_{s-1}]$, $(x_s,x_{s-1})$, $(x_s,x_0)$, $T[z'', x_1] \cup (x_1,x_0)$, $T[z'', x] \cup (x,y)$ and $T[z',y]$ is a $S$-$C(k,1,1,1,1,1)$ whenever $ y \notin T_{x_s}$, and the union of $T[x_s, y]$, $T[z,x] \cup (x,y)$, $T[z,x_2]$, $(x_1,x_2)$, $(x_1,x_0)$ and $(x_s,x_0)$ is a $S$-$C(k,1,1,1,1,1)$ whenever $y \in T_{x_{s}} $, a contradiction. If $z_{s-1} \leqslant _T z'$ and $z_{s-1} \neq z'$, then the union of $T[z,x_0]$, $(x_s,x_0)$, $(x_s,x_{s-1})$, $T[z', x_{s-1}]$, $T[z',y]$, and $T[z,x] \cup (x,y)$ is a $S$-$C(k,1,1,1,1,1)$, a contradiction. This ends the proof of this assertion.  $\hfill {\lozenge}$
 \end{proof}
 \begin{assertion} 
There exists no arc $(x,y) \in A(D_i^3)$, such that $x \in T_{x_2} \backslash \lbrace x_2 \rbrace$ and $y \in T_{z} \backslash (T_{x_1} \cup T_{x_2})$, for any $z \leqslant_T x_2$.
   \end{assertion}   
    \begin{proof}
    Assume to the contrary that such an arc exists. If $x \leqslant_T x_0$, then the union of $T[x,x_0]$, $(x_1,x_0)$, $(x_1,x_2)$, $T[z,x_2]$, $T[z,y]$ and $(x,y)$ is a $S$-$C(k,1,1,1,1,1)$, a contradiction. So $x$ is not an ancestor of $x_0$. Let $z'=\textrm{l.c.a}(x,x_0)$. If $z' \neq x_2$, then the union of $T[z,y]$, $T[z',x] \cup (x,y)$, $T[z', x_0]$, $(x_1,x_0)$, $(x_1,x_2)$ and $T[z,x_2]$ is a $S$-$C(k,1,1,1,1,1)$, a contradiction. If $z' = x_2$, then the union of $T[z'',y]$, $T[x_2,x] \cup (x,y)$, $T[x_2, x_{s-1}]$, $(x_s,x_{s-1})$, $(x_s,x_0)$ and $T[z'',x_1] \cup (x_1,x_0)$ is a $S$-$C(k,1,1,1,1,1)$ whenever $x$ and $x_{s-1}$ are not ancestors where $z''=\textrm{l.c.a}(x_1,y)$, and  the union of $T[z,x_0]$, $(x_s,x_0)$, $(x_s,x_{s-1})$, $T[x, x_{s-1}]$, $(x,y)$, and $T[z,y] $ is a $S$-$C(k,1,1,1,1,1)$ whenever $x \leqslant_T x_{s-1}$, and the union of $T[z,y] $, $ (x_s,x_{s-1}) \cup T[x_{s-1},x] \cup (x,y)$, $(x_s,x_0)$, $ (x_1,x_0)$, $(x_1,x_2)$ and $T[z,x_2]$ is a $S$-$C(k,1,1,1,1,1)$ whenever $x \in T_{x_{s-1}}$, a contradiction. $\hfill {\lozenge}$
  \end{proof}     
    \begin{assertion} 
    There exists no arc $(x,y) \in A(D_i^3)$, such that  $x \in T_{z} \backslash (T_{x_1} \cup T_{x_2})$ and $y\in T_{x_1} \backslash \lbrace x_1 \rbrace $, for any $z \leqslant_T x_2$.
   \end{assertion}   
   \begin{proof}
    Assume to the contrary that such an arc exists. Then  the union of $T[z,x_{s-1}]$, $(x_s,x_{s-1})$, $(x_s,x_0)$, $(x_1,x_0)$, $T[x_1,y]$ and $T[z,x] \cup (x,y)$ is a $S$-$C(k,1,1,1,1,1)$ whenever $x$ is not ancestor of $x_0$, and the union of $T[x,x_{s-1}]$, $(x_s,x_{s-1})$, $(x_s,x_0)$, $(x_1,x_0)$, $T[x_1,y]$ and $ (x,y)$ is a $S$-$C(k,1,1,1,1,1)$ whenever $x \leqslant_T x_0$, a contradiction.$\hfill {\lozenge}$
   \end{proof}
   \begin{assertion}
 For every  $z \leqslant_T x_2$, there exists no  arc $(x,y) \in A(D_i^3)$, such that $x\in T_{x_1} \backslash \lbrace x_1 \rbrace $ and  $y \in T_{z} \backslash (T_{x_1} \cup T_{x_2})$, where $y$ is not an ancestor of $x_2$.
   \end{assertion}   
   \begin{proof}
  Assume otherwise. Then the union of $T[z,x_{s-1}]$, $(x_s,x_{s-1})$, $(x_s,x_0)$, $(x_1,x_0)$, $T[x_1,x] \cup (x,y)$ and $T[z,y]$ is a $S$-$C(k,1,1,1,1,1)$, a contradiction. $\hfill {\lozenge}$
   \end{proof}\\
   
   But the antidirected cycle $C$ has to be completed, then due to the previous assertions, there exists an arc $(x,y) \in A(C)$ and an arc $(x,y') \in A(C)$, such that $x\in T_{x_1} \backslash \lbrace x_1 \rbrace $ , $y \in T]z_1,x_2]$ and $y' \in T_{x_1} \cup T_{x_2}$. If $y' \in T_{x_1}$, then the union of $T[x_1,y']$, $(x,y')$, $(x,y) \cup T[y,x_{s-1}]$, $(x_s,x_{s-1})$, $(x_s,x_0)$ and $(x_1,x_0)$ is a $S$-$C(k,1,1,1,1,1)$, a contradiction. If $y' \in T]x_2,x_{s-1}[$, then the union of $(x,y') \cup T[y',x_{s-1}]$, $(x_s,x_{s-1})$, $(x_s,x_0)$, $(x_1,x_0)$, $(x_1,x_2)$ and $(x,y) \cup T[y,x_2]$ is a $S$-$C(k,1,1,1,1,1)$, a contradiction. If $y' \in T_{x_{s-1}}$, then the union of $T[z_s,x_0]$, $(x_1,x_0)$, $(x_1,x_2)$, $(x,y) \cup  T[y,x_2]$, $(x,y')$ and $T[z_s,x_s] \cup (x_s,x_{s-1}) \cup T[x_{s-1},y']$ is a $S$-$C(k,1,1,1,1,1)$, a contradiction. So, $y'$ is not an ancestor with $x_{s-1}$, then let $z=\textrm{l.c.a}(y',x_{s-1})$. So the union of $T[x_1,x] \cup (x,y')$, $T[z,y']$, $T[z,x_{s-1}]$, $(x_s,x_{s-1})$, $(x_s,x_0)$ and $(x_1,x_0)$ is a $S$-$C(k,1,1,1,1,1)$ whenever $y' \notin T_{x_s}$, and the union of $T[x_s,y']$, $(x,y')$, $(x,y)$, $T[z_1,y]$, $T[z_1,x_1] \cup (x_1,x_0)$ and $(x_s,x_0)$ is a $S$-$C(k,1,1,1,1,1)$ whenever $y' \in T_{x_s}$, a contradiction.\\
   Thus, the antidirected cycle can't be closed, a contradiction. So $x_2 \notin T]z_1,z_{s-1}]$. This terminates the proof of Lemma \ref{lem3}.$\hfill {\square}$
    \end{proof} 
    \begin{corollary}\label{cor1}
     Suppose that $D_i^3$ contains an antidirected cycle $C:=x_0x_1 \dots x_s$, where $x_0$ is a sink of $C$ of maximal level. For all $ i \in \lbrace 1,2, \dots ,s \rbrace $, let $z_i=\mathrm{l.c.a}(x_i,x_0)$ such that $z_1 \leqslant_T z_{s-1} \leqslant_T z_s $ and $z_1$, $z_{s-1}$ and $z_s$ are pairwisely distinct. Then $s=3$.
     \end{corollary}
     \begin{proof}
    Assume that $s \geq 5$.  Lemmas \ref{lem1}, \ref{lem2} and \ref{lem3} implies that $x_2$ and $x_{s-1}$  are not ancestors. 
Let $z=\textrm{l.c.a}(x_2,x_{s-1})$. If $z \leqslant _T z_{s-1}$, then the union of $T[z,x_{s-1}]$, $(x_s,x_{s-1})$, $(x_s,x_0)$, $(x_1,x_0)$, $(x_1,x_2)$ and $T[z,x_2]$ is a $S$-$C(k,1,1,1,1,1)$ whenever $x_2 \notin T_{x_s}$, and the union of $T[x_s,x_2]$, $(x_1,x_2)$, $(x_1,x_0)$, $T[z_{s-1},x_0]$, $T[z_{s-1},x_{s-1}]$ and $(x_s,x_{s-1})$ is a $S$-$C(k,1,1,1,1,1)$ whenever $x_2 \in T_{x_s}$, a contradiction. If $  z_{s-1} \leqslant _T z$ and  $  z_{s-1} \neq  z$, then  the union of $T[z_s,x_0]$, $(x_1,x_0)$, $(x_1,x_2)$, $T[z,x_2]$, $T[z,x_{s-1}]$ and $T[z_s,x_s] \cup (x_s,x_{s-1})$ is a $S$-$C(k,1,1,1,1,1)$, a contradiction.
  This ends the proof of Corollary \ref{cor1}.$\hfill {\square}$
     \end{proof}\vspace{3mm}\\
     
 Now we consider the case when $z_1 = z_s  $, such that $x_1 \leqslant_T x_s$ and $x_{s-1} \leqslant_T x_0$. We prove that whatever the position of $x_2$ on $T$ is, $D$ contains a $S$-$C(k,1,1,1,1,1)$ in case $l(C)\geq 8$.
     \begin{lemma}\label{lem12}
   Suppose that $D_i^3$ contains an antidirected cycle $C:=x_0x_1\dots x_s$ of length at least $8$, where $x_0$ is a sink of $C$ of maximal level. Let $z_i=\mathrm{l.c.a}(x_i,x_0)$, for all $ i \in \lbrace 1,2,\dots ,s \rbrace $, such that $x_{s-1} \leqslant_T x_0$ and  $z_1 = z_{s} $, and $x_1 \leqslant_T x_s$. Then $s \leq 5$.
 \end{lemma} 
   \begin{proof}
  Assume to the contrary that $s\geq 7$. We study according to the position of $x_2$. Assume first that $x_2$ and $x_{s-1}$ are not ancestors. Then the union of $T[z_2,x_2]$, $(x_1,x_2)$, $(x_1,x_0)$, $(x_s,x_0)$, $(x_s,x_{s-1})$ and $T[z_2,x_{s-1}]$ is a $S$-$C(k,1,1,1,1,1)$ whenever $l(T[z_2,x_2]) \geq k$ or $l(T[z_2,x_{s-1}]) \geq k$. So $l(T[z_2,x_2]) =l(T[z_2,x_{s-1}])< k$, and thus $x_3$ and $x_{s-1}$ are not ancestors. The union of $T[z_3,x_{s-1}]$, $(x_s,x_{s-1})$, $(x_s,x_0)$, $(x_1,x_0)$, $(x_1,x_2)$ and $T[z_3,x_3] \cup (x_3,x_2)$ is a $S$-$C(k,1,1,1,1,1)$ whenever $x_3 \notin T_{x_1}$, and by a similar manner to the proof of Lemma \ref{lem1}, we can show that if $x_3 \in T_{x_s}$, then $V(C) \backslash \lbrace x_0,x_1,x_s,x_{s-1} \rbrace \subseteq T_{x_2} \cup T_{x_3}$, which leads to a contradiction. So,  $x_3 \in T_{x_1} \backslash T_{x_s}$, and thus the union of $T[x_1,x_3] \cup (x_3,x_2)$, $T[z_2,x_2]$, $T[z_2,x_{s-1}]$, $(x_s,x_{s-1})$, $(x_s,x_0)$ and $(x_1,x_0)$ is a $S$-$C(k,1,1,1,1,1)$, a contradiction. Thus $x_2$ and $x_{s-1}$ are ancestors.
  
  By a similar manner to the proof of Lemma \ref{lem3}, one can show that $x_2$ is not an ancestor of $x_{s-1}$. If $x_2 \in T_{x_{s-1}} \backslash T[x_{s-1},x_0]$, then by a similar manner to the proof of Lemma \ref{lem1}, we can show that $x_3 \notin T_{x_s}$ and $x_3 \notin T_{x_1} \backslash ( T_{x_s} \cup T[x_1,x_s])$. Also, we can easily show that $x_3 \notin T[x_1,x_s]$. Hence $x_3 \notin T_{x_1}$. If $x_3 \notin T_{x_{s-1}}$, and so the union of $T[z_3,x_{s-1}]$, $(x_s,x_{s-1})$, $(x_s,x_0)$, $(x_1,x_0)$, $(x_1,x_2)$ and $T[z_3,x_3] \cup (x_3,x_2)$ is a $S$-$C(k,1,1,1,1,1)$ whenever $l(T[z_3,x_3]) \geq k$ or $l(T[z_3,x_{s-1}]) \geq k$. Hence $l(T[z_3,x_3]) =l(T[z_3,x_{s-1}]) < k$. Then  by a similar manner to the proof of Lemma \ref{lem1}, we can show that $x_4 \notin T_{x_{s-1}} \backslash T[x_{s-1},x_0]$, and also one can easily check that $D$ contains a $S$-$C(k,1,1,1,1,1)$, whatever the position of $x_4$ in $T$ is, a contradiction. So, $x_3 \in T_{x_{s-1}}$. But then the union of $T[x_3,x_0]$, $(x_s,x_0)$, $(x_s,x_{s-1})$, $T[z_s,x_{s-1}]$, $T[z_s,x_1] \cup (x_1,x_2)$ and $(x_3,x_2)$ is a $S$-$C(k,1,1,1,1,1)$ whenever $x_3 \leqslant_T x_0$, and  by a similar manner to the proof of Lemma \ref{lem1}, we can show that $x_3 \notin T_{x_{s-1}} \backslash T[x_{s-1},x_0]$, a contradiction. So $x_2 \notin T_{x_{s-1}} \backslash T[x_{s-1},x_0]$. Hence $x_2 \in T]x_{s-1},x_0[$.\\
 Now, we discuss according to the position of $x_3$.  
  
  \textbf{Case 1:} If $x_3 \in T_{x_{s-1}}$.\\
   If $x_4 \notin T_{x_{s-1}}$, we get that $x_4$ is not an ancestor of $x_0$. So, the union of $(x_3,x_2) \cup T[x_2,x_0]$, $(x_s,x_0)$, $(x_s,x_{s-1})$, $T[z_4,x_{s-1}]$, $T[z_4,x_4]$ and $(x_3,x_4)$ is a $S$-$C(k,1,1,1,1,1)$ whenever $x_4 \notin T_{x_s}$, and the union of $T[x_s,x_4]$, $(x_3,x_4)$, $(x_3,x_2)$, $(x_1,x_2)$, $(x_1,x_0)$ and $(x_s,x_0)$ is a $S$-$C(k,1,1,1,1,1)$ whenever $x_4 \in T_{x_s}$, a contradiction. So, $x_4 \in T_{x_{s-1}}$. If $x_4 \in T_{x_{s-1}} \backslash T_{x_2}$, then the union of $(x_s,x_{s-1}) \cup T[x_{s-1}, x_4]$, $(x_3,x_4)$, $(x_3,x_2)$, $(x_1,x_2)$, $(x_1,x_0)$ and $(x_s,x_0)$ is a $S$-$C(k,1,1,1,1,1)$, a contradiction. So, $x_4 \in T_{x_2}$, and so the union of $(x_3,x_4) \cup T[x_4,x_0]$, $(x_s,x_0)$, $(x_s,x_{s-1})$, $T[z_s,x_{s-1}]$, $T[z_s,x_1] \cup (x_1,x_2)$ and $(x_3,x_2)$ is a $S$-$C(k,1,1,1,1,1)$ whenever $x_4 \leqslant_T x_0$. Now, by a similar manner to the proof of Lemma \ref{lem1}, one can show that $x_4 \notin T_{x_2} \backslash T[x_2,x_0]$, a contradiction. So, $x_3 \notin T_{x_{s-1}}$.
  
  \textbf{Case 2:} If $x_3 \in T_{x_s}$.\\
   If $x_4$ and $x_{s-1}$ are not ancestors, then the union of $T[x_1,x_s] \cup (x_s,x_{s-1})$, $T[z_4,x_{s-1}]$, $T[z_4,x_4]$, $(x_3,x_4)$, $(x_3,x_2)$ and $(x_1,x_2)$ is a $S$-$C(k,1,1,1,1,1)$ whenever $x_4 \notin T_{x_1}$, and the union of $T[z_s,x_4]$, $(x_3,x_4)$, $(x_3,x_2) \cup T[x_2,x_0]$, $(x_s,x_0)$, $(x_s,x_{s-1})$ and $T[z_s,x_{s-1}]$ is a $S$-$C(k,1,1,1,1,1)$ whenever $x_4 \in T_{x_1} \backslash T_{x_s}$, and the union of $T[x_s,x_4]$, $(x_3,x_4)$, $(x_3,x_2)$, $(x_1,x_2)$, $(x_1,x_0)$ and $(x_s,x_0)$ is a $S$-$C(k,1,1,1,1,1)$ whenever $x_4 \in T_{x_s}$, a contradiction. If $x_4 \leqslant_T x_{s-1}$, then the union of $(x_3,x_4) \cup T[x_4,x_{s-1}]$, $(x_s,x_{s-1})$, $(x_s,x_0)$, $(x_1,x_0)$, $(x_1,x_2)$ and $(x_3,x_2)$ is a $S$-$C(k,1,1,1,1,1)$, a contradiction. If $x_4 \in T_{x_{s-1}} $, then the union of $(x_s,x_{s-1}) \cup T [x_{s-1},x_4]$, $(x_3,x_4)$, $(x_3,x_2)$, $(x_1,x_2)$, $(x_1,x_0)$ and $(x_s,x_0)$ is a $S$-$C(k,1,1,1,1,1)$ whenever $x_4 \in T_{x_{s-1}} \backslash T_{x_2}$, and the union of $T[z_s,x_{s-1}]$, $(x_s,x_{s-1})$, $(x_s,x_0)$, $(x_3,x_4) \cup T[x_4,x_0]$, $(x_3,x_2)$ and $T[z_s,x_1] \cup (x_1,x_2)$ is a $S$-$C(k,1,1,1,1,1)$ whenever $x_4 \in T]x_2,x_0[$, and $T[z_4,x_4]$, $T[x_s,x_3] \cup (x_3,x_4)$, $(x_s,x_{s-1})$, $T[z_s,x_{s-1}]$, $T[z_s,x_1] \cup (x_1,x_0)$ and $T[z_4,x_0]$ is a $S$-$C(k,1,1,1,1,1)$ whenever $x_4 \in T_{x_2} \backslash T[x_2,x_0]$, a contradiction. So, $x_3 \notin T_{x_s}$. 
  
 \textbf{Case 3:} If $x_3 \in T_{x_1} \backslash T_{x_s}$.\\
  If $x_4$ and $x_{s-1}$ are not ancestors, then the union of $(x_3,x_2) \cup T[x_2,x_0]$, $(x_s,x_0)$, $(x_s,x_{s-1})$, $T[z_4,x_{s-1}]$, $T[z_4,x_4]$ and $(x_3,x_4)$  is a $S$-$C(k,1,1,1,1,1)$ whenever $x_4 \notin T_{x_s}$, and the union of $T[x_s,x_4]$, $(x_3,x_4)$, $(x_3,x_2)$, $(x_1,x_2)$, $(x_1,x_0)$ and $(x_s,x_0)$  is a $S$-$C(k,1,1,1,1,1)$ whenever $x_4 \in T_{x_s}$, a contradiction. If $x_4 \leqslant_T x_{s-1}$, we obtain that the union of $(x_3,x_4) \cup T[x_4,x_{s-1}]$, $(x_s,x_{s-1})$, $(x_s,x_0)$, $(x_1,x_0)$, $(x_1,x_2)$ and $(x_3,x_2)$  is a $S$-$C(k,1,1,1,1,1)$, a contradiction. If $x_4 \in T_{x_{s-1}}$, then the union of $(x_s,x_{s-1}) \cup T[x_{s-1},x_4]$, $(x_3,x_4)$, $(x_3,x_2)$, $(x_1,x_2)$, $(x_1,x_0)$ and $(x_s,x_0)$  is a $S$-$C(k,1,1,1,1,1)$ whenever $x_4 \in T_{x_{s-1}} \backslash T_{x_2}$. So, $x_4 \in T_{x_2}$, then the union of $(x_3,x_4) \cup T[x_4,x_0]$, $(x_s,x_0)$, $(x_s,x_{s-1})$, $T[z_s,x_{s-1}]$, $T[z_s,x_1] \cup (x_1,x_2)$ and $(x_3,x_2)$  is a $S$-$C(k,1,1,1,1,1)$ whenever $x_4 \leqslant_T x_0$, and the union of $T[z_s,x_3] \cup (x_3,x_4)$, $T[z_4,x_4]$, $T[z_4,x_0]$, $(x_s,x_0)$, $(x_s,x_{s-1})$ and $T[z_s, x_{s-1}]$  is a $S$-$C(k,1,1,1,1,1)$ whenever $x_4 $ is not an ancestor of $x_0$, a contradiction. Hence $x_3 \notin T_{x_1}$.
   
  \textbf{Case 4:} If $x_3 \in T_{z_3} \backslash (T_{x_1} \cup T_{x_{s-1}})$.\\ 
  In this case, the union of $T[z_3,x_{s-1}]$, $(x_s,x_{s-1})$, $(x_s,x_0)$, $(x_1,x_0)$, $(x_1,x_2)$ and $T[z_3,x_3] \cup (x_3,x_2)$  is a $S$-$C(k,1,1,1,1,1)$ whenever $l(T[z_3,x_{s-1}]) \geq k$ or $l(T[z_3,x_{3}]) \geq k$. Hence $l(T[z_3,x_{s-1}])=l(T[z_3,x_{3}]) < k$. If $x_4 \notin T_{x_{s-1}}$, then  the union of $T[x_1,x_s] \cup (x_s,x_{s-1})$, $T[z_4,x_{s-1}]$, $T[z_4,x_4]$, $(x_3,x_4)$, $(x_3,x_2)$ and $(x_1,x_2)$  is a $S$-$C(k,1,1,1,1,1)$ whenever $x_4 \notin T_{x_1}$, and the union of $T[x_1,x_4]$, $(x_3,x_4)$, $(x_3,x_2)$, $(x_s,x_{s-1}) \cup T[x_{s-1},x_2]$, $(x_s,x_0)$ and $(x_1,x_0)$  is a $S$-$C(k,1,1,1,1,1)$ whenever $x_4 \in T_{x_1} \backslash T_{x_s}$, and the union of $T[x_s,x_4]$, $(x_3,x_4)$, $(x_3,x_2)$, $(x_1,x_2)$, $(x_1,x_0)$ and $(x_s,x_0)$  is a $S$-$C(k,1,1,1,1,1)$ whenever $x_4 \in T_{x_s}$, a contradiction. So $x_4 \in T_{x_{s-1}}$. Then the union of $(x_s,x_{s-1}) \cup  T[x_{s-1} ,x_4]$, $(x_3,x_4)$, $(x_3,x_2)$, $(x_1,x_2)$, $(x_1,x_0)$ and $(x_s,x_0)$  is a $S$-$C(k,1,1,1,1,1)$ whenever $x_4 \in T_{x_{s-1}} \backslash T_{x_2}$, and the union of $(x_3,x_4) \cup T[x_4,x_0]$, $(x_s,x_0)$, $(x_s,x_{s-1})$, $T[z_s,x_{s-1}]$, $T[z_s,x_1] \cup (x_1,x_2)$ and $(x_3,x_2)$  is a $S$-$C(k,1,1,1,1,1)$ whenever $x_2 \leqslant_T x_4 \leqslant_T x_0$, and the union of $T[x_1,x_s] \cup (x_s,x_{s-1})$, $T[z_3,x_{s-1}]$, $T[z_3,x_3] \cup (x_3,x_4)$, $T[z_4,x_4]$, $T[z_4,x_0]$ and $(x_1,x_0)$  is a $S$-$C(k,1,1,1,1,1)$ whenever $x_4 $ is not ancestor of $x_0$, a contradiction.\\
   Thus,  $x_2 \notin T]x_{s-1},x_0[$. 
  This ends the proof of this Lemma.$\hfill {\square}$\vspace{3mm}\\
   \end{proof}

We are ready to color $D_i^3$: 
 \begin{proposition}\label{p3}
 $ \chi (D_i^3) \leq 24$, for all $i \in \lbrace 1,2,\dots ,k \rbrace$.	
 \end{proposition}
\begin{proof} Assume to the contrary that $\chi (D_i^3) \geq 25 $ for some $ i \in \lbrace 1,2,\dots ,k \rbrace$. Then by Theorem \ref{antidi}, $D_i^3$ contains an antidirected cycle $C$ of length at least 8. Let $C:=x_0x_1\dots x_s$ and assume without loss of generality that $x_0$ is a sink of $C$ of maximal level in $T$. Let $z_i=\textrm{l.c.a}(x_0,x_i)$, for all $i \in \lbrace 1,2,\dots ,s \rbrace$, and without loss of generality suppose that $z_1 \leqslant_T z_s$, and if $z_1=z_s$, then assume that $l_T(x_s) \geq l_T(x_1)$. Due to the above discussion, we obtain that $s \leq 5$, a contradiction. This terminates the proof of Proposition \ref{p3}.
$\hfill {\square}$
\end{proof}  
 
\subsection{Main Theorem}
\vspace{1.5mm} Now we are ready to state our main theorem:
 \begin{theorem}
 		Let $D$ be a strongly connected digraph having no subdivisions of  $C(k,1,1,1,1,1)$, then the chromatic number of $D$ is at most $7\cdot 24^2\cdot k$.
 \end{theorem}
\begin{proof}
Since $D$ is strongly connected digraph, then there is a  spanning out-tree $T$ of $D$.  By Proposition \ref{finaltree}, we may assume that $T$ is final. Define $D^i_j$ as before for $i \in \lbrace 1,2,\dots ,k \rbrace$ and $j  \in \lbrace 1,2,3 \rbrace$. Due to Lemma \ref{far} together with Proposition \ref{hello}, Proposition \ref{prop2} and Proposition \ref{p3},  we get that $\chi(D_i)\leq 7\cdot 24^2$ for all $i \in  \lbrace 1,2,\dots ,k \rbrace$.  As $V(D)=\bigcup _{i=1}^{k} V(D_i)$,  by assigning $7\cdot 24^2$ distinct colors to each $D_i$, we obtain a proper coloring of $D$ with $7\cdot 24^2\cdot k$ colors. $\hfill {\square}$
\end{proof}
\begin{theorem}
Let $D$ be a strongly connected digraph having no subdivisions of  $C(1,1,1,1,1,1)$, then the chromatic number of $D$ is at most $7\cdot 16^2\cdot k$.
\end{theorem}
\begin{proof}
Let $T$ be a  spanning out-tree of $D$.  Proposition \ref{finaltree} implies that  we can assume that $T$ is final. Define $D^i_j$ as before for $i \in \lbrace 1,2,\dots ,k \rbrace$ and $j \in \lbrace 1,2,3 \rbrace$. Let $i \in \lbrace 1,2,\dots ,k \rbrace$ and let $t \in \{1,3\}$. If $\chi(D_i^t)\geq 17$, then Theorem \ref{antidi} implies that $D_{i}^{t}$ contains an antidirected cycle $C$ of length at least $6$. If $C$ is of length $6$, then $C$ is a $S$-$C(1,1,1,1,1,1)$, a contradiction. Then the length of $C$ is at least $8$, and so the proof of Proposition \ref{hello} and Proposition \ref{p3} implies that $D$ contains a subdivisions of  $C(1,1,1,1,1,1)$, a contradiction. Hence $\chi(D_i^t)\leq 16$ for $t=1,3$. Due to Lemma \ref{far} together with Proposition \ref{prop2},  we get that $\chi(D_i)\leq 7\cdot 16^2$ for all $i \in \lbrace 1,2,\dots ,k \rbrace$.  As $V(D)=\bigcup _{i=1}^{k} V(D_i)$,  by assigning $7\cdot 16^2$ distinct colors to each $D_i$, we obtain a proper coloring of $D$ with $7\cdot 16^2\cdot k$ colors. $\hfill {\square}$
\end{proof}
\section*{Statements and Declarations}
\textbf{Conflict of interest} The authors have not disclosed any competing interests.\vspace{2mm}\\
\textbf{Data sharing} Data sharing not applicable to this article as no datasets were generated or analysed during
 the current study.\vspace{2mm}\\
\textbf{Funds} The authors declare that no funds, grants, or other support were received during the preparation of this manuscript.

 	\end{document}